\documentclass[a4paper,11pt]{article}

\usepackage[small,bf,hang]{caption} 

\usepackage[T1]{fontenc}
\usepackage{todonotes}
\usepackage[ruled,linesnumbered]{algorithm2e}
\usepackage{amsmath,amsthm,amssymb}
\usepackage{graphicx,subcaption,tikz}
\usepackage{comment,xspace,paralist}		\usepackage{authblk}
\usepackage[shortlabels]{enumitem}	
\usepackage{hyperref}
\usepackage{changepage}
\hypersetup{
    colorlinks=true,
    citecolor = blue,
    linkcolor=black,
    filecolor=magenta,
    urlcolor=cyan,
    bookmarks=true,
}
\usepackage[margin=1in]{geometry}
\usepackage{verbatim}
\usepackage[capitalize]{cleveref}
\usepackage{color}

\usetikzlibrary{decorations.pathreplacing, calligraphy}

\newtheorem{theorem}{Theorem}

\usepackage{thmtools}
\newlist{lemmalist}{enumerate}{1}
\setlist[lemmalist]{label=(\roman{lemmalisti}),
	ref=\thelemma:$(\roman{lemmalisti})$,
	noitemsep}
\declaretheorem[
name=Lemma]{lemma}

\Crefname{lemmalist}{Lemma}{Lemmas}
\addtotheorempostheadhook[lemma]{\crefalias{lemmalisti}{lemmalist}}

\newtheorem{corollary}{Corollary}

\newenvironment{proof2}[1][]%
{\noindent{\setcounter{equation}{0}\it Proof.
}{#1}{}}{$\diamond$\vspace{0ex}}

\title{{\bf There are finitely many $5$-vertex-critical $(P_6,\text{bull})$-free graphs}}
\author[a,b]{Yiao Ju}
\author[c]{Jorik Jooken}
\author[c,d]{Jan Goedgebeur}
\author[e]{Shenwei Huang\thanks{Email: shenweihuang@nankai.edu.cn.}}

\affil[a]{College of Computer Science, Nankai University, Tianjin 300071, China}
\affil[b]{Tianjin Key Laboratory of Network and Data Security Technology, Nankai University, Tianjin 300071, China}
\affil[c]{Department of Computer Science, KU Leuven Campus Kulak-Kortrijk, 8500 Kortrijk, Belgium}
\affil[d]{Department of Applied Mathematics, Computer Science and Statistics, Ghent University, 9000 Ghent, Belgium}
\affil[e]{School of Mathematical Sciences and LPMC, Nankai University, Tianjin 300071, China}

\begin{document}

\maketitle

\begin{abstract}
In this paper, we are interested in $4$-colouring algorithms for graphs that do not contain an induced path on $6$ vertices nor an induced bull, i.e., the graph with vertex set $\{v_1,v_2,v_3,v_4,v_5\}$ and edge set $\{v_1v_2,v_2v_3,v_3v_4,v_2v_5,v_3v_5\}$. Such graphs are referred to as $(P_6,\text{bull})$-free graphs. A graph $G$ is \emph{$k$-vertex-critical} if $\chi(G)=k$, and every proper induced subgraph $H$ of $G$ has $\chi(H)<k$. In the current paper, we investigate the structure of $5$-vertex-critical $(P_6,\text{bull})$-free graphs and show that there are only finitely many such graphs, thereby answering a question of Maffray and Pastor. A direct corollary of this is that there exists a polynomial-time algorithm to decide if a $(P_6,\text{bull})$-free graph is $4$-colourable such that this algorithm can also provide a certificate that can be verified in polynomial time and serves as a proof of 4-colourability or non-4-colourability.

{\bf Keywords.} Graph colouring; Critical graphs; Strong perfect graph theorem;  Hereditary graph classes; Polynomial-time algorithms

\end{abstract}

\section{Introduction}

 All graphs in this paper are finite and simple. We follow~\cite{BM08} for general graph theory terminology that is not defined here. A \emph{$k$-colouring} of a graph $G$ is a function $\phi:V(G)\rightarrow\mathcal{C}$, where $\mathcal{C}$ is a set of $k$ colours, such that each pair of adjacent vertices $u,v\in V(G)$ has $\phi(u)\neq\phi(v)$. A graph $G$ is \emph{$k$-colourable} if $G$ admits a $k$-colouring. The \emph {chromatic number} of $G$, denoted by $\chi(G)$, is the minimum number $k$ such that $G$ is $k$-colourable. A graph $G$ is \emph{$k$-chromatic} if $\chi(G)=k$.

For any fixed $k$, the $k$-colouring problem asks whether an input graph $G$ is $k$-colourable. It is one of the most important and widely investigated problems in graph theory, both from an algorithmic as well as a combinatorial point of view. The problem can be solved in polynomial time for $k \leq 2$, but it is one of the prototypical NP-complete problems for $k \geq 3$. The complexity situation can be very different when considering restrictions on the input graphs. One such restriction that attracted considerable attention in the literature on $k$-colouring is to forbid a small set of graphs from occurring as induced subgraphs of the input graph.  We refer the interested reader to~\cite{GJPS17} for a survey of this rich and active research area. For two graphs $G$ and $H$, $G$ is \emph{$H$-free} if $G$ has no induced subgraphs that are isomorphic to $H$. For a family $\mathcal{H}$ of graphs, $G$ is \emph{$\mathcal{H}$-free} if $G$ is $H$-free for every $H\in\mathcal{H}$. We often write $(H_1,\ldots,H_s)$-free for $\{H_1,\ldots,H_s\}$-free. Note that for any family of graphs $\mathcal{H}$, the class of $\mathcal{H}$-free graphs is \emph{hereditary}, i.e., it is closed under vertex deletion. 

Let $P_n$, $C_n$ and $K_n$ denote the path, cycle and complete graph on $n$ vertices, respectively. The \emph{bull} is the graph with vertex set $\{v_1,v_2,v_3,v_4,v_5\}$ and edge set $\{v_1v_2,v_2v_3,v_3v_4,v_2v_5,v_3v_5\}$, the \emph{banner} is the graph with vertex set $\{v_1,v_2,v_3,v_4,v_5\}$ and edge set $\{v_1v_2,v_2v_3,v_3v_4,v_1v_4,v_1v_5\}$ and the \emph{chair} is the graph with vertex set $\{v_1,v_2,v_3,v_4,v_5\}$ and edge set $\{v_1v_2,v_2v_3,v_1v_4,v_1v_5\}$ (see~\cref{fig:P6BullBannerChair}). For an integer $k \geq 1$ and a graph $H$ the graph $kH$ denotes the disjoint union of $k$ copies of the graph $H$. The complement of a graph $G$ is denoted by $\overline{G}$. 

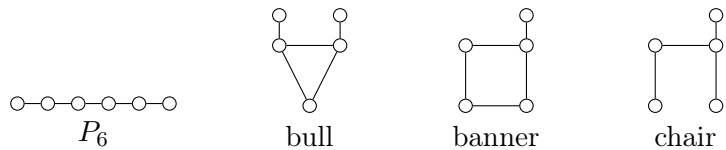
\begin{figure}[h!]
\centering
\begin{tikzpicture}[scale=0.4]
\tikzstyle{vertex}=[circle, draw, fill=white, inner sep=1pt, minimum size=5pt]

	\node at (2.5,-1) {$P_6$};
	\node[vertex](1) at (0,0) {};
	\node[vertex](2) at (1,0) {};
	\node[vertex](3) at (2,0) {};
   	\node[vertex](4) at (3,0) {};
	\node[vertex](5) at (4,0) {};
	\node[vertex](6) at (5,0) {};

    \foreach \from/\to in {1/2,2/3,3/4,4/5,5/6}
		\draw (\from) -- (\to);
\end{tikzpicture}
\hspace{10mm}
\begin{tikzpicture}[scale=0.4]
\tikzstyle{vertex}=[circle, draw, fill=white, inner sep=1pt, minimum size=5pt]

	\node at (0,-1) {bull};
	\node[vertex](1) at (0,0) {};
	\node[vertex](2) at (1,2) {};
	\node[vertex](3) at (-1,2) {};
   	\node[vertex](4) at (1,3) {};
	\node[vertex](5) at (-1,3) {};

    \foreach \from/\to in {1/2,1/3,2/3,2/4,3/5}
		\draw (\from) -- (\to);
\end{tikzpicture}
\hspace{10mm}
\begin{tikzpicture}[scale=0.4]
\tikzstyle{vertex}=[circle, draw, fill=white, inner sep=1pt, minimum size=5pt]

	\node at (0,-1) {banner};
	\node[vertex](1) at (-1,0) {};
	\node[vertex](2) at (1,2) {};
	\node[vertex](3) at (-1,2) {};
   	\node[vertex](4) at (1,3) {};
	\node[vertex](5) at (1,0) {};

    \foreach \from/\to in {1/5,1/3,2/3,2/5,2/4}
		\draw (\from) -- (\to);
\end{tikzpicture}
\hspace{10mm}
\begin{tikzpicture}[scale=0.4]
\tikzstyle{vertex}=[circle, draw, fill=white, inner sep=1pt, minimum size=5pt]

	\node at (0,-1) {chair};
	\node[vertex](1) at (-1,0) {};
	\node[vertex](2) at (1,2) {};
	\node[vertex](3) at (-1,2) {};
   	\node[vertex](4) at (1,3) {};
	\node[vertex](5) at (1,0) {};

    \foreach \from/\to in {1/3,2/3,2/5,2/4}
		\draw (\from) -- (\to);
\end{tikzpicture}
\caption{The graphs $P_6$, bull, banner and chair.}\label{fig:P6BullBannerChair}
\end{figure}

The fourth author proved in~\cite{H16} that the 4-colouring problem remains NP-complete for $P_7$-free graphs, but is polynomial-time solvable for $(P_6,\text{banner})$-free graphs. Moreover, he conjectured that the 4-colouring problem can be solved in polynomial time for $P_6$-free graphs. This conjecture was further supported by Chudnovsky et al.~\cite{CMSZ17}, Brause et al.~\cite{BSHVK15} and Maffray and Pastor~\cite{MP17} who showed that the problem can be solved in polynomial time for $(P_6,C_5)$-free, $(P_6,\text{chair})$-free and $(P_6,\text{bull})$-free graphs, respectively. The conjecture was finally proven by Chudnovsky, Spirkl and Zhong in 2019~\cite{CSZ19}.

The previously discussed papers provide algorithms that can produce a 4-colouring if it exists, but are otherwise only able to report that it does not exist without being able to show a small \textit{certificate} as a proof thereof, i.e., these algorithms can provide positive certificates, but not negative certificates. We call an algorithm for a decision problem \textit{certifying} if together with its answer it outputs a certificate that allows one to verify in polynomial time that the output of the algorithm is indeed correct (see~\cite{MMNS11} for a more thorough treatment of this subject). In other words, the algorithms from the previous paragraph are not certifying algorithms. Such certifying algorithms are important to verify that an algorithm really produced the right result and was not corrupted by an implementation bug or a hardware failure. This consideration naturally leads to the following notion: a graph $G$ is \emph{$k$-vertex-critical} if $\chi(G)=k$, and every proper induced subgraph $H$ of $G$ has $\chi(H)<k$. Note that a graph $G$ is not $k$-colourable if and only if $G$ contains a $(k+1)$-vertex-critical graph as an induced subgraph. Hence, such an induced subgraph can serve as a negative certificate. In case this negative certificate comes from a finite list, one can indeed verify in polynomial time whether it occurs as an induced subgraph of the input graph.

Ho{\`a}ng et al.~\cite{HMRSV15} showed that there are infinitely many $5$-vertex-critical $P_5$-free graphs (note that these graphs are of course also $P_6$-free). Maffray and Pastor stated in their paper~\cite{MP17} that they do not know whether there are only finitely many $5$-vertex-critical $(P_6,\text{bull})$-free graphs and called this ``an interesting question''. In the current paper, we show that there are indeed only finitely many such graphs. This has the following algorithmic consequence.

\begin{corollary}
There exists a polynomial-time certifying algorithm to determine whether a $(P_6,\text{bull})$-free graph is $4$-colourable.
\end{corollary}
\begin{proof}
Let $G$ be the $(P_6,\text{bull})$-free graph from the input and $\mathcal{L}$ be the finite list of $5$-vertex-critical $(P_6,bull)$-free graphs. The certifying algorithm will first run the polynomial-time algorithm from Maffray and Pastor~\cite{MP17} and output \lq $4$-colourable\rq~together with a 4-colouring in case $G$ is $4$-colourable. Otherwise, the algorithm will iterate over each $5$-vertex-critical $(P_6,\text{bull})$-free graph $G' \in \mathcal{L}$ and output \lq not $4$-colourable\rq~together with $G'$ in case $G'$ occurs as an induced subgraph of $G$. Note that $\mathcal{L}$ is finite and therefore there exists a constant $c$ such that $|V(G')| \leq c$ for each $G' \in \mathcal{L}$ and therefore one can check whether $G'$ occurs as an induced subgraph of $G$ in polynomial time.
\end{proof}

The rest of this paper is structured as follows: in~\cref{sec:pre} we introduce some additional notation and preliminary lemmas that will be used later. In~\cref{sec:algo} we discuss an algorithm that will be used to prove several claims about the structure of $5$-vertex-critical $(P_6,\text{bull})$-free graphs with the aid of a computer. Finally in~\cref{sec:proofs} we present the proof of the main theorem of this paper, stating that there are finitely many $5$-vertex-critical $(P_6,\text{bull})$-free graphs.

\section{Preliminaries}\label{sec:pre}

 The \emph{clique number} of $G$, denoted by $\omega(G)$, is the largest number $k$ such that $G$ contains an induced $K_k$. A graph $G$ is $\emph perfect$ if every induced subgraph $H$ of $G$ has $\chi(H)=\omega(H)$. Obviously, $K_k$ is the only perfect $k$-vertex-critical graph. Let $G$ be a graph and $X,Y\subseteq V(G)$. We write $G[X]$ for the subgraph of $G$ induced by $X$. We say that $X$ induces an $H$ if $G[X]$ is isomorphic to $H$. We say that $A \subseteq X$ is a \textit{component} of $X$ if $A$ induces a connected component of $G[X]$. The \textit{neighbourhood} of $X$, denoted by $N(X)$, is the set of vertices in $V(G)\setminus X$ that have a neighbour in $X$. If $X=\{u\}$, we simply write $N(u)$ instead of $N(\{u\})$. Let $N_Y(X)=N(X)\cap Y$. $X$ is $\emph complete$ ($\emph anti$-$\emph complete$) to $Y$, if every vertex in $X$ is adjacent (nonadjacent, resp.) to every vertex in $Y$. $X$ is $\emph pure$ to $Y$ if $X$ is complete or anti-complete to $Y$, and $\emph mixed$ on $Y$ otherwise. $X$ is a $\emph homogeneous$ set of $G$ if every vertex in $V(G)\setminus X$ is pure to $X$. A homogeneous set $X$ of $G$ is \emph{nontrivial} if $2\leq|X|\leq|V(G)|-1$. Two vertices $u,v \in V(G)$ are \emph{comparable} if $u \neq v$ and $N(u) \subseteq N(v)$ or $N(v) \subseteq N(u)$. We use the following preliminaries.

\begin{lemma}[\cite{XJGH25}]\label{lem:homo-critical}
Let $G$ be a $k$-vertex-critical graph, and $S$ be a homogeneous set in $G$. Then $S$ is $\chi(S)$-vertex-critical.
\end{lemma}

\begin{lemma}[\cite{BM08}]\label{lem:critical}
If $G$ is a $k$-vertex-critical graph, then $G$ is connected, has no clique cutsets or comparable vertices.
\end{lemma}

We can strengthen the lemma that says that $k$-vertex-critical graphs have no comparable vertices as follows.

\begin{lemma}[\cite{CGHS21}]\label{lem:comparable}
Let $G$ be a $k$-vertex-critical graph. Then $G$ has no two nonempty disjoint subsets $X$ and $Y$ of $V(G)$ that satisfy all the following conditions.

$\bullet$ $X$ is anti-complete to $Y$.

$\bullet$ $\chi(X)\leq \chi(Y)$.

$\bullet$ $Y$ is complete to $N(X)$.
\end{lemma}

We will also make use of the following theorems.

\begin{theorem}[Strong Perfect Graph Theorem~\cite{CRST06}]\label{thm:SPGT}
A graph $G$ is perfect if and only if neither $G$ nor $\overline{G}$ contains an induced odd cycle on 5 or more vertices.
\end{theorem}

\begin{theorem}[\cite{CGSZ20}]\label{thm:4-critical-P6}
There are exactly 80 4-vertex-critical $P_6$-free graphs.
\end{theorem}

Let $G_1$ and $G_2$ be graphs with disjoint vertex sets, and $S$ be a homogeneous set of $G_1$. We say that $G$ is obtained from $G_1$ by $\emph substituting$ $G_2$ for $S$ if:

$\bullet$ $V(G)=(V(G_1)\setminus S)\cup V(G_2)$,

$\bullet$ $u,v\in V(G_1)\setminus S$ are adjacent in $G$ if and only if $u,v$ are adjacent in $G_1$,

$\bullet$ $u,v\in V(G_2)$ are adjacent in $G$ if and only if $u,v$ are adjacent in $G_2$,

$\bullet$ $u\in V(G_1)\setminus S$ and $v\in V(G_2)$ are adjacent in $G$ if and only if $u$ is complete to $S$ in $G_1$.

One can observe that $V(G_2)$ is a homogeneous set of $G$, and $G_1$ is the graph obtained from $G$ by substituting $G_1[S]$ for $V(G_2)$. We now prove the following lemmas related to substituting.

\begin{lemma}\label{lem:substitute-leading}
Let $G_1$ and $G_2$ be graphs with disjoint vertex sets, and $S$ be a homogeneous set of $G_1$ with $\chi(G_2)\leq \chi(G_1[S])$. Let $G$ be the graph obtained from $G_1$ by substituting $G_2$ for $S$, then $\chi(G)\leq \chi(G_1)$.
\end{lemma}
\begin{proof}
Consider any $\chi(G_1)$-colouring of $G_1$. There are at least $\chi(G_1[S])$ colours used in $S$, and these colours do not appear in $N_{G_1}(S)=N_{G}(G_2)$. Since $\chi(G_2)\leq \chi(G_1[S])$, we may use the colours used in $S$ to colour $G_2$, and let the vertices in $G_1-S$ keep their colours in $G_1$. This gives a $\chi(G_1)$-colouring of $G$.
\end{proof}

\begin{lemma}\label{lem:substitute}
Let $G_1$ be a $p$-vertex-critical graph, $S$ be a homogeneous set of $G_1$ with $\chi(G_1[S])=q$, and $G_2$ be a $q$-vertex-critical graph. Let $G$ be the graph obtained from $G_1$ by substituting $G_2$ for $S$. Then $G$ is $p$-vertex-critical.
\end{lemma}
\begin{proof}

First we prove that $\chi(G)=p$. Since $G_1$ is the graph obtained from $G$ by substituting $G_1[S]$ for $V(G_2)$, and $\chi(G_1[S])=\chi(G_2)=q$, by using \cref{lem:substitute-leading} twice, we have $\chi(G)\leq \chi(G_1)$ and $\chi(G_1)\leq \chi(G)$. So $\chi(G)=\chi(G_1)=p$.


Then we prove that for every vertex $v\in V(G)$, $G-v$ is $(p-1)$-colourable. If $v$ is in $G_1-S$, then $G-v$ is the graph obtained from $G_1-v$ by substituting $G_2$ for $S$. By \cref{lem:substitute-leading} we have $\chi(G-v)\leq\chi(G_1-v)$. And we have $\chi(G_1-v)=p-1$ because $G_1$ is $p$-vertex-critical. If $v$ is in $G_2$, then $\chi(G_2-v)=q-1$ because $G_2$ is $q$-vertex-critical. Let $u\in S$, then $\chi(G_1[S\setminus\{u\}])=q-1$ by \cref{lem:homo-critical}. Since $G-v$ is the graph obtained from $G_1-u$ by substituting $G_2-v$ for $S\setminus\{u\}$, and $\chi(G_1[S\setminus\{u\}])=\chi(G_2-v)$, by \cref{lem:substitute-leading} we have $\chi(G-v)\leq \chi(G_1-u)=p-1$.
\end{proof}

\section{Generation algorithm}\label{sec:algo}
In a series of papers~\cite{GJORS24,GS18,HMRSV15,XJGH25,XJGH24}, an algorithm was developed, fine-tuned and generalised to recursively generate all $k$-vertex-critical $\mathcal{H}$-free graphs that contain the graph $I$ as an induced subgraph (for input parameters $k$, $\mathcal{H}$ and $I$). Hence, if the algorithm terminates, this serves as a computational proof that there are only finitely many $k$-vertex-critical $\mathcal{H}$-free graphs containing $I$ as an induced subgraph (since the algorithm is exhaustive). On the other hand, if the algorithm does not terminate there could either be finitely or infinitely many such graphs. When the algorithm is executed with $k=5$, $\mathcal{H}=\{P_6,\text{bull}\}$ and $I=K_1$, the algorithm does not terminate after running for one week and will likely never terminate even though there are only finitely many 5-vertex-critical $(P_6,\text{bull})$-free graphs (as we will show later). Instead, in the current paper we will use this algorithm to prove several claims about the structure of 5-vertex-critical $(P_6,\text{bull})$-free graphs, which will finally help us to prove that there are only finitely many such graphs. We make the source code of the algorithms that were used to verify these claims publicly available at~\cite{jorik-github}. To keep the current paper self-contained, we sketch the main ideas behind this algorithm, but refer the interested reader to~\cite{GS18} for a much more thorough description of the algorithm.

To generate all graphs such that $I$ occurs as an induced subgraph, one can simply start from $I$ and recursively add one vertex and add edges between the new vertex and the old vertices in all possible ways. Here, it is important to notice that multiple ways of adding a set of edges could lead to isomorphic graphs and it is not necessary to recurse for each of these possibilities, since the recursion would again lead to isomorphic copies. Therefore, the algorithm computes the canonical form of its input graph using the \textit{nauty} package~\cite{MP14} and stores this in an appropriate data structure that allows for fast insertion and lookup. This is useful, because two graphs have the same canonical form if and only if these graphs are isomorphic. In case the algorithm finds a graph which is isomorphic with a graph that it considered before, it does not recurse. However, if one is only interested in those graphs which are $k$-vertex-critical and $\mathcal{H}$-free, one can further restrict the possibilities for the edges that are added by employing additional pruning rules. For example, if adding a set of edges leads to a graph that is not $\mathcal{H}$-free or is not $k$-colourable, that graph cannot occur as a proper induced subgraph of any $k$-vertex-critical $\mathcal{H}$-free graph and so the algorithm can first verify whether that graph is an $\mathcal{H}$-free $k$-vertex-critical graph (and output it if so) and then consider the next set of edges to be added without recursing for the current set of edges. However, more sophisticated pruning rules are possible by relying on the structure of $k$-vertex-critical graphs. For example,~\cref{lem:critical} says that a $k$-vertex-critical graph $G$ cannot contain two vertices $u, v$ such that $N(u) \subseteq N(v)$. In particular, this means that if $I$ is an induced subgraph of $G$ such that $N_{V(I)}(u) \subseteq N_{V(I)}(v)$, there must exist a vertex $w \in V(G) \setminus V(I)$ such that $uw \in E(G)$, but $vw \notin E(G)$. Hence, if the current graph considered by the algorithm is $I$, it may add the vertex $w$ and only add those edges between $w$ and $V(I)$ such that $uw \in E(G)$, but $vw \notin E(G)$. The full set of pruning rules used by the algorithm is much larger than the set of pruning rules sketched in this paragraph and we refer the interested reader to~\cite{GS18} for a more complete overview.

\section{Proofs}\label{sec:proofs}

A homogeneous $C_5$ is a homogeneous set that induces a $C_5$. We first prove that there are finitely many 5-vertex-critical ($P_6$,bull)-free graphs with no homogeneous $C_5$ and then use this result to prove that there are finitely many 5-vertex-critical ($P_6$,bull)-free graphs.

\begin{theorem}\label{thm:no-homo-C5-finite}
There are finitely many 5-vertex-critical ($P_6$,bull)-free graphs with no homogeneous $C_5$.
\end{theorem}
\begin{proof}
Let $G$ be a 5-vertex-critical ($P_6$,bull)-free graph with no homogeneous $C_5$. Since $K_5$ and $\overline{C_9}$ are 5-vertex-critical, we may assume that $G$ is $(K_5, \overline{C_9})$-free. By \cref{thm:SPGT} and since $G$ is $P_6$-free, $G$ contains an induced  $C_5$ or $\overline{C_7}$. We run the program for 5-vertex-critical ($P_6$, bull)-free graphs that contain an induced $\overline{C_7}$. The program terminates and outputs all 28 $5$-vertex-critical ($P_6$, bull)-free graphs that contain an induced $\overline{C_7}$. So we may assume that $G$ contains an induced $C_5$. 

Let $Q=\{u_1, u_2, u_3, u_4, u_5\}$ be an induced $C_5$ in $G$ with $u_iu_{i+1}\in E(G)$ for $i=1,\ldots,5$, with all indices modulo 5. For each $i=1,\ldots,5$, let

$\bullet$ $S_1(i)=\{v\in V(G)\setminus Q: N_Q(v)=\{u_i\}\}$,

$\bullet$ $S_2(i)=\{v\in V(G)\setminus Q: N_Q(v)=\{u_{i-1},u_{i+1}\}\}$,

$\bullet$ $S_3(i)=\{v\in V(G)\setminus Q: N_Q(v)=\{u_{i-1},u_i,u_{i+1}\}\}$,

$\bullet$ $S_4(i)=\{v\in V(G)\setminus Q: N_Q(v)=\{u_{i-2},u_{i-1},u_{i+1},u_{i+2}\}\}$

and let

$\bullet$ $S_5=\{v\in V(G)\setminus Q: N_Q(v)=Q\}$,

$\bullet$ $S_0=\{v\in V(G)\setminus Q: N_Q(v)=\emptyset\}$.

Let $S_j=\bigcup_{i=1}^5 S_j(i)$ for $j=1,2,3,4$ (and note that every vertex in $S_j$ has exactly $j$ neighbours in the set $Q$).

In the following, we discuss the properties of these vertex sets, and finally prove that $|V(G)|$ is bounded by bounding each of these vertex sets. Since $Q$ is an arbitrary $C_5$ in $G$, these properties hold when any induced $C_5$ in $G$ is treated as $Q$. If we say that we may assume that $G$ has some property, we mean that there are finitely many 5-vertex-critical ($P_6$,bull)-free graphs with no homogeneous $C_5$ that do not satisfy the property.

\begin{enumerate}[label=(\arabic*),series=edu]

\item\label{clm:homo-chi}
We may assume that if $X\subset V(G)$ is a homogeneous set of $G$, then $\chi(G[X])\leq 3$ and $G[X]$ is a clique.

\begin{proof2}
Since $G$ is 5-vertex-critical, $\chi(G[X])\leq 4$. By \cref{lem:homo-critical}, $X$ is $\chi(G[X])$-vertex-critical. If $\chi(G[X])=1$ or $2$, then $G[X]$ is a $K_1$ or $K_2$, respectively. If $\chi(G[X])=3$, then $G[X]$ is an odd cycle, and so is a $K_3$ since $G$ has no homogeneous $C_5$ and is $P_6$-free. If $\chi(G[X])=4$, since $G$ is connected by \cref{lem:critical}, there is a vertex $v\in N(X)$. Since $X$ is homogeneous, $v$ is complete to $X$, and so $\chi(G[X\cup \{v\}])=5$. So $V(G)=X\cup \{v\}$. Since there are finitely many 4-vertex-critical ($P_6$,bull)-free graphs by \cref{thm:4-critical-P6}, we may assume that $G$ has no 4-chromatic homogeneous sets.
\end{proof2}

\item\label{clm:whole}
$V(G)=Q\cup S_1\cup S_2\cup S_3\cup S_4\cup S_5\cup S_0$.

\begin{proof2}
It suffices to prove that for any $v\in V(G)\setminus Q$, $N_Q(v)$ cannot be $\{u_i,u_{i+1}\}$ or $\{u_i,u_{i+1},u_{i+3}\}$ for any $i$. If $N_Q(v)=\{u_i,u_{i+1}\}$ or $\{u_i,u_{i+1},u_{i+3}\}$, then $\{u_{i-1},u_i,u_{i+1},u_{i+2},v\}$ induces a bull.
\end{proof2}


\item\label{clm:0-134}
$S_0$ is anti-complete to $S_1\cup S_3\cup S_4$.

\begin{proof2}
Let $v_1\in S_0$. If $v_2\in S_1(i)$ is adjacent to $v_1$, then $\{v_1,v_2,u_i,u_{i+1},u_{i+2},u_{i+3}\}$ induces a $P_6$. If $v_3\in S_3(i)\cup S_4(i+2)$ is adjacent to $v_1$, then $v_3u_i,v_3u_{i+1}\in E(G)$ and $v_3u_{i+2}\notin E(G)$, and then $\{v_1,v_3,u_{i+1},u_{i+2},u_i\}$ induces a bull.
\end{proof2}

\item\label{clm:0-2}
Each component of $S_0$ is pure to each component of $S_2(i)$.

\begin{proof2}
If $v_1\in S_0$ is mixed on an edge $v_2v_3$ in $S_2(i)$, then $\{v_1,v_2,v_3,u_{i+1},u_{i+2}\}$ induces a bull. If $v_4\in S_2(i)$ is mixed on an edge $v_5v_6$ in $S_0$, then $\{v_5,v_6,v_4,u_{i+1},u_{i+2},u_{i+3}\}$ induces a $P_6$.
\end{proof2}

\item\label{clm:5-stable}
We may assume that $S_5$ is a stable set.

\begin{proof2}
If there is a $K_2$ in $S_5$, then $G$ contains an induced $C_5\vee K_2$, which is 5-vertex-critical.
\end{proof2}

\item\label{clm:5-12}
$S_5$ is complete to $S_1\cup S_2$.

\begin{proof2}
If $v_1\in S_5$ and $v_2\in S_1(i)\cup S_2(i+1)$ are nonadjacent, then $\{v_2,u_i,v_1,u_{i+3},u_{i+1}\}$ induces a bull.
\end{proof2}

\item\label{clm:1-1}
$S_1(i)$ is complete to $S_1(i+2)\cup S_1(i-2)$ and anti-complete to $S_1(i+1)\cup S_1(i-1)$.

\begin{proof2}
By symmetry, it suffices to prove that $S_1(i)$ is complete to $S_1(i+2)$ and anti-complete to $S_1(i+1)$. Let $v_1\in S_1(i)$. If $v_2\in S_1(i+2)$ is nonadjacent to $v_1$, then $\{v_1,u_i,u_{i+4},u_{i+3},u_{i+2},v_2\}$ induces a $P_6$. If $v_3\in S_1(i+1)$ is adjacent to $v_1$, then \\ $\{v_1,v_3,u_{i+1},u_{i+2},u_{i+3},u_{i+4}\}$ induces a $P_6$.
\end{proof2}


\item\label{clm:1-2+0}
Each component of $S_1(i)$ is pure to each component of $S_2(i)$.

\begin{proof2}
If $v_1\in S_1(i)$ is mixed to an edge $v_2v_3$ in $S_2(i)$, then $\{v_1,v_2,v_3,u_{i+1},u_{i+2}\}$ induces a bull. If $v_4\in S_2(i)$ is mixed to an edge $v_5v_6$ in $S_1(i)$, then $\{v_5,v_6,v_4,u_{i+1},u_{i+2},u_{i+3}\}$ induces a $P_6$.
\end{proof2}

\item\label{clm:1-2+1}
$S_1(i)$ is anti-complete to $S_2(i+1)\cup S_2(i-1)$.

\begin{proof2}
By symmetry it suffices to prove that $S_1(i)$ is anti-complete to $S_2(i+1)$. If $v_1\in S_1(i)$ and $v_2\in S_2(i+1)$ are adjacent, then $\{u_{i-1},u_i,v_2,u_{i+2},v_1\}$ induces a bull.
\end{proof2}

\item\label{clm:1-2+2}
$S_1(i)$ is pure to each component of $S_2(i+2)$ and pure to each component of $S_2(i-2)$.

\begin{proof2}
By symmetry it suffices to prove that $S_1(i)$ is pure to each component of $S_2(i+2)$. If $v_1\in S_1(i)$ is mixed on an edge $v_2v_3\in S_2(i+2)$, then $\{v_1,v_2,v_3,u_{i+1},u_{i+2}\}$ induces a bull. Let $v_4,v_5\in S_1(i)$ and $v_6\in S_2(i+2)$, and let $v_4v_6\in E(G)$ and $v_5v_6\notin E(G)$. If $v_4v_5\in E(G)$, then $\{v_6,v_4,u_i,u_{i-1},v_5\}$ induces a bull. If $v_4v_5\notin E(G)$, then $\{v_5,u_i,v_4,v_6,u_{i+3},u_{i+2}\}$ induces a $P_6$.
\end{proof2}

\item\label{clm:1-3}
$S_1(i)$ is complete to $S_3(i)$ and anti-complete to $S_3\setminus S_3(i)$.

\begin{proof2}
By symmetry it suffices to prove that $S_1(i)$ is complete to $S_3(i)$ and anti-complete to $S_3(i+1)\cup S_3(i+2)$. Let $v_1\in S_1(i)$. If $v_2\in S_3(i)$ is nonadjacent to $v_1$, then $\{v_1,u_i,u_{i+1},u_{i+2},v_2\}$ induces a bull. If $v_3\in S_3(i+1)$ is adjacent to $v_1$, then $\{u_{i-1},u_i,v_3,u_{i+2},v_1\}$ induces a bull. If $v_4\in S_3(i+2)$ is adjacent to $v_1$, then $\{v_1,v_4,u_{i+3},u_{i+4},u_{i+2}\}$ induces a bull.
\end{proof2}

\item\label{clm:1-4+12}
If $S_1(i)$ is nonempty, then $S_4\setminus S_4(i)$ is empty.

\begin{proof2}
By symmetry it suffices to prove that if $S_1(i)$ is nonempty, then $S_4(i+1)\cup S_4(i+2)$ is empty. Let $v_1\in S_1(i)$. If there exists $v_2\in S_4(i+1)$, then $\{u_{i+1},u_i,v_2,u_{i+3},v_1\}$ induces a bull if $v_1v_2\in E(G)$, or $\{v_1,u_i,v_2,u_{i+2},u_{i-1}\}$ induces a bull if $v_1v_2\notin E(G)$. If there exists $v_3\in S_4(i+2)$, then $\{v_1,v_3,u_{i+3},u_{i+2},u_{i+4}\}$ induces a bull if $v_1v_3\in E(G)$, or $\{v_1,u_i,u_{i+1},u_{i+2},v_3\}$ induces a bull if $v_1v_3\notin E(G)$.
\end{proof2}

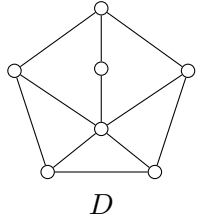
\begin{figure}[h!]
\centering
\begin{tikzpicture}[scale=0.4]
\tikzstyle{vertex}=[circle, draw, fill=white, inner sep=1pt, minimum size=5pt]

	\node at (0,-3.5) {$D$};
	\node[vertex](1) at (0,3) {};
	\node[vertex](2) at (-{3*cos(18)},{3*sin(18}) {};
	\node[vertex](3) at (-{3*sin(36)},-{3*cos(36}) {};
   	\node[vertex](4) at ({3*sin(36)},-{3*cos(36}) {};
	\node[vertex](5) at ({3*cos(18)},{3*sin(18}) {};
	\node[vertex](6) at (0,1) {};
        \node[vertex](7) at (0,-1) {};
        
    \foreach \from/\to in {1/2,2/3,3/4,4/5,1/5,1/6,2/7,3/7,4/7,5/7,6/7}
		\draw (\from) -- (\to);
\end{tikzpicture}
\caption{The graph $D$.}\label{fig:D}
\end{figure}

\item\label{clm:1-4+0}
$S_1(i)$ is anti-complete to $S_4(i)$.

\begin{proof2}
Let $D$ be the 7-vertex graph with $Q$, and two vertices $v_1\in S_1(1)$ and $v_2\in S_4(1)$, and make $v_1$ and $v_2$ adjacent (see \cref{fig:D}). We run the program for 5-vertex-critical ($P_6$, bull)-free graphs that contain an induced $D$. The program terminates and outputs 0 graphs.
\end{proof2}

Combining \ref{clm:whole}, \ref{clm:0-134} and \ref{clm:5-12}-\ref{clm:1-4+0} we have the following \ref{clm:1-homo} and \ref{clm:1-private}.

\item\label{clm:1-homo}
Each component of $S_1(i)$ is homogeneous.

\item\label{clm:1-private}
Let $v_1,v_2$ be in different components of $S_1(i)$. If $v_3\in V(G)\setminus S_1(i)$ is adjacent to $v_1$ and nonadjacent to $v_2$, then $v_3\in S_2(i)$.

\item\label{clm:2-2+1}
Each component of $S_2(i)$ is pure to each component of $S_2(i+1)$ and pure to each component of $S_2(i-1)$.

\begin{proof2}
By symmetry it suffices to prove that each vertex of $S_2(i)$ is pure to each component of $S_2(i+1)$. If $v_1\in S_2(i)$ is mixed on an edge $v_2v_3$ in $S_2(i+1)$, then $\{v_1,v_2,v_3,u_{i+2},u_{i+3}\}$ induces a bull.
\end{proof2}

\item\label{clm:2-2+2}
$S_2(i)$ is anti-complete to $S_2(i+2)\cup S_2(i-2)$.

\begin{proof2}
By symmetry it suffices to prove that $S_2(i)$ is anti-complete to $S_2(i+2)$. If $v_1\in S_2(i)$ and $v_2\in S_2(i+2)$ are adjacent, then $\{u_{i-1},v_1,u_{i+1},u_{i+2},v_2\}$ induces a bull.
\end{proof2}

\item\label{clm:2-3+12}
$S_2(i)$ is complete to $S_3(i+1)\cup S_3(i-1)$ and anti-complete to $S_3(i+2)\cup S_3(i-2)$.

\begin{proof2}
By symmetry it suffices to prove that $S_2(i)$ is complete to $S_3(i+1)$ and anti-complete to $S_3(i+2)$. Let $v_1\in S_2(i)$. If $v_2\in S_3(i+1)$ is nonadjacent to $v_1$, then $\{v_1,u_{i+1},u_{i+2},u_{i+3},v_2\}$ induces a bull. If $v_3\in S_3(i+2)$ is adjacent to $v_1$, then $\{u_i,u_{i+1},v_3,u_{i+3},v_1\}$ induces a bull.
\end{proof2}

\item\label{clm:2-4+12}
$S_2(i)$ is complete to $S_4\setminus S_4(i)$.

\begin{proof2}
By symmetry it suffices to prove that $S_2(i)$ is complete to $S_4(i+1)\cup S_4(i+3)$. If $v_1\in S_2(i)$ and $v_2\in S_4(i+1)\cup S_4(i+3)$ are nonadjacent, then $\{v_1,u_{i+4},v_2,u_{i+2},u_i\}$ induces a bull.
\end{proof2}

\begin{figure}[h!]
\centering
\begin{tikzpicture}[scale=0.35]
\tikzstyle{vertex}=[circle, draw, fill=white, inner sep=1pt, minimum size=5pt]

	\node at (0,-3.5) {$H_1$};
	\node[vertex](1) at (0,3) {};
	\node[vertex](2) at (-{3*cos(18)},{3*sin(18}) {};
	\node[vertex](3) at (-{3*sin(36)},-{3*cos(36}) {};
   	\node[vertex](4) at ({3*sin(36)},-{3*cos(36}) {};
	\node[vertex](5) at ({3*cos(18)},{3*sin(18}) {};
	\node[vertex](6) at (0,1) {};
        \node[vertex](7) at (0,-1) {};
        
    \foreach \from/\to in {1/2,2/3,3/4,4/5,1/5,2/6,5/6,2/7,3/7,4/7,5/7,6/7}
		\draw (\from) -- (\to);
\end{tikzpicture}
\hspace{10mm}
\begin{tikzpicture}[scale=0.35]
\tikzstyle{vertex}=[circle, draw, fill=white, inner sep=1pt, minimum size=5pt]

	\node at (0,-3.5) {$H_2$};
	\node[vertex](1) at (0,3) {};
	\node[vertex](2) at (-{3*cos(18)},{3*sin(18}) {};
	\node[vertex](3) at (-{3*sin(36)},-{3*cos(36}) {};
   	\node[vertex](4) at ({3*sin(36)},-{3*cos(36}) {};
	\node[vertex](5) at ({3*cos(18)},{3*sin(18}) {};
	\node[vertex](6) at (0,1) {};
        \node[vertex](7) at (0,-1) {};
        
    \foreach \from/\to in {1/2,2/3,3/4,4/5,1/5,1/6,3/6,2/7,3/7,4/7,5/7,6/7}
		\draw (\from) -- (\to);
\end{tikzpicture}
\hspace{10mm}
\begin{tikzpicture}[scale=0.35]
\tikzstyle{vertex}=[circle, draw, fill=white, inner sep=1pt, minimum size=5pt]

	\node at (0,-3.5) {$H_3$};
	\node[vertex](1) at (0,3) {};
	\node[vertex](2) at (-{3*cos(18)},{3*sin(18}) {};
	\node[vertex](3) at (-{3*sin(36)},-{3*cos(36}) {};
   	\node[vertex](4) at ({3*sin(36)},-{3*cos(36}) {};
	\node[vertex](5) at ({3*cos(18)},{3*sin(18}) {};
	\node[vertex](6) at (0,1) {};
        \node[vertex](7) at (0,-1) {};
        
    \foreach \from/\to in {1/2,2/3,3/4,4/5,1/5,2/6,4/6,2/7,3/7,4/7,5/7,6/7}
		\draw (\from) -- (\to);
\end{tikzpicture}
\hspace{10mm}
\begin{tikzpicture}[scale=0.35]
\tikzstyle{vertex}=[circle, draw, fill=white, inner sep=1pt, minimum size=5pt]

	\node at (0,-3.5) {$H_4$};
	\node[vertex](1) at (0,3) {};
	\node[vertex](2) at (-{3*cos(18)},{3*sin(18}) {};
	\node[vertex](3) at (-{3*sin(36)},-{3*cos(36}) {};
   	\node[vertex](4) at ({3*sin(36)},-{3*cos(36}) {};
	\node[vertex](5) at ({3*cos(18)},{3*sin(18}) {};
	\node[vertex](6) at (0,1) {};
        \node[vertex](7) at (0,-1) {};
        
    \foreach \from/\to in {1/2,2/3,3/4,4/5,1/5,3/6,5/6,2/7,3/7,4/7,5/7,6/7}
		\draw (\from) -- (\to);
\end{tikzpicture}
\hspace{10mm}
\begin{tikzpicture}[scale=0.35]
\tikzstyle{vertex}=[circle, draw, fill=white, inner sep=1pt, minimum size=5pt]

	\node at (0,-3.5) {$H_5$};
	\node[vertex](1) at (0,3) {};
	\node[vertex](2) at (-{3*cos(18)},{3*sin(18}) {};
	\node[vertex](3) at (-{3*sin(36)},-{3*cos(36}) {};
   	\node[vertex](4) at ({3*sin(36)},-{3*cos(36}) {};
	\node[vertex](5) at ({3*cos(18)},{3*sin(18}) {};
	\node[vertex](6) at (0,1) {};
        \node[vertex](7) at (0,-1) {};
        
    \foreach \from/\to in {1/2,2/3,3/4,4/5,1/5,4/6,1/6,2/7,3/7,4/7,5/7,6/7}
		\draw (\from) -- (\to);
\end{tikzpicture}
\caption{The graphs $H_1$, $H_2$, $H_3$, $H_4$ and $H_5$.}\label{fig:Hi}
\end{figure}
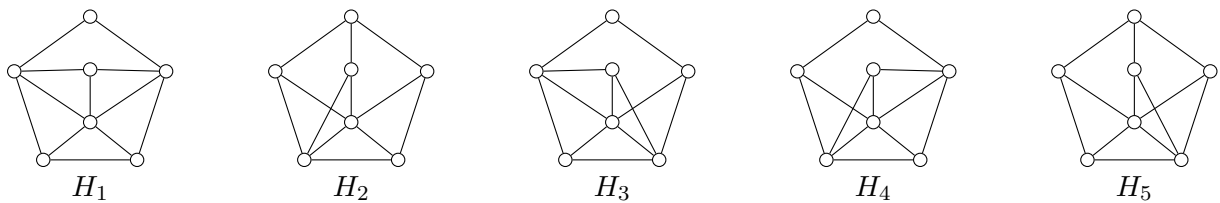

\item\label{clm:2-4+0}
$S_2(i)$ is anti-complete to $S_4$.

\begin{proof2}
For $i=1,2,3,4,5$, let $H_i$ be the 7-vertex graph with $Q$, and two vertices $v_1\in S_2(1)$ and $v_2\in S_4(i)$, and make $v_1$ and $v_2$ adjacent (see \cref{fig:Hi}). Note that $H_1$ contains an induced bull, so it cannot occur as induced subgraph of any bull-free graph. Moreover, $H_2$ is isomorphic with $H_5$ and $H_3$ is isomorphic with $H_4$. We run the program for 5-vertex-critical ($P_6$, bull)-free graphs that contain an induced $H_i$ ($i \in \{2,3\}$). For $i=2$ and $i=3$, the program terminates and outputs 0 graphs.
\end{proof2}

Combining \ref{clm:2-4+12} and \ref{clm:2-4+0} we have the following \ref{clm:2-4+12-2}.

\item\label{clm:2-4+12-2}
If $S_2(i)$ is nonempty, then $S_4\setminus S_4(i)$ is empty. 

Combining \ref{clm:whole},  \ref{clm:0-2}, \ref{clm:5-12}, \ref{clm:1-2+0}-\ref{clm:1-2+2}, and \ref{clm:2-2+1}-\ref{clm:2-4+12-2}, we have the following \ref{clm:2-homo}.

\item\label{clm:2-homo}
Let $X$ be a component of $S_2(i)$. If $X$ is anti-complete to $S_3(i)$, then $X$ is homogeneous.

\item\label{clm:3-3+1}
$S_3(i)$ is complete to $S_3(i+1)\cup S_3(i-1)$.

\begin{proof2}
By symmetry it suffices to prove that $S_3(i)$ is complete to $S_3(i+1)$. If $v_1\in S_3(i)$ and $v_2\in S_3(i+1)$ are nonadjacent, then $\{v_1,u_{i+1},u_{i+2},u_{i+3},v_2\}$ induces a bull.
\end{proof2}

\begin{figure}[h!]
\centering
\begin{tikzpicture}[scale=0.4]
\tikzstyle{vertex}=[circle, draw, fill=white, inner sep=1pt, minimum size=5pt]

	\node at (0,-3.5) {$E$};
	\node[vertex](1) at (0,3) {};
	\node[vertex](2) at (-{3*cos(18)},{3*sin(18}) {};
	\node[vertex](3) at (-{3*sin(36)},-{3*cos(36}) {};
   	\node[vertex](4) at ({3*sin(36)},-{3*cos(36}) {};
	\node[vertex](5) at ({3*cos(18)},{3*sin(18}) {};
	\node[vertex](6) at (0,1) {};
        \node[vertex](7) at (0,-1) {};
        
    \foreach \from/\to in {1/2,2/3,3/4,4/5,1/5,2/6,3/6,4/6,5/6,2/7,3/7,4/7,5/7}
		\draw (\from) -- (\to);
\end{tikzpicture}
\hspace{10mm}
\caption{The graph $E$.}\label{fig:E}
\end{figure}
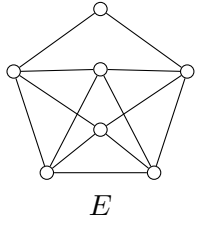

\item\label{clm:4-clique}
Each $S_4(i)$ is a clique.

\begin{proof2}
Let $E$ be the 7-vertex graph with $Q$, and a $2K_1$ in $S_4(1)$. We run the program for 5-vertex-critical ($P_6$, bull)-free graphs that contain an induced $E$. The program terminates and outputs 0 graphs.
\end{proof2}

\item\label{clm:4-size}
$|S_4(i)|\leq 2$ for each $1 \leq i \leq 5$.

\begin{proof2}
Suppose not, by \ref{clm:4-clique}, there is a $K_3$ in $S_4(i)$, which together with $\{u_{i+1},u_{i+2}\}$ induces a $K_5$.
\end{proof2}

\item\label{clm:5+0-34}
If $v_1\in S_5$ has a neighbour in $S_0$, then $v_1$ is complete to $S_3\cup S_4$.

\begin{proof2}
Let $v_2\in S_0$ be adjacent to $v_1$. Suppose that $v_3\in S_3(i)\cup S_4(i+2)$ is nonadjacent to $v_1$. Then $v_3u_{i+1}\in E(G)$ and $v_3u_{i+2}\notin E(G)$. By \ref{clm:0-134}, $v_2v_3\notin E(G)$. Therefore, $\{v_2,v_1,u_{i+1},v_3,u_{i+2}\}$ induces a bull.
\end{proof2}

\item\label{clm:0+2-5}
If $v_1\in S_0$ has a neighbour in $S_2(i)$, then $v_1$ is complete to $S_5$.

\begin{proof2}
Let $v_2\in S_2(i)$ be adjacent to $v_1$. Suppose that $v_3\in S_5$ is nonadjacent to $v_1$. By \ref{clm:5-12}, $v_2v_3\in E(G)$. Then $\{v_1,v_2,v_3,u_{i+3},u_{i+1}\}$ induces a bull.
\end{proof2}

\item\label{clm:5or0empty}
At least one of $S_5$ and $S_0$ is empty.

\begin{proof2}
Suppose not, then $S_0$ and $S_5$ are nonempty. First we suppose that every component of $S_0$ has a neighbour in $S_2$. By \ref{clm:0-2} and \ref{clm:0+2-5}, every component of $S_0$ is complete to $S_5$. By \ref{clm:5-12} and \ref{clm:5+0-34}, $S_5$ is complete to $S_1\cup S_2\cup S_3\cup S_4$. So $V(G)\setminus S_5$ is complete to $S_5$. Since $S_5$ is a stable set by \ref{clm:5-stable}, and since $G$ is 5-vertex-critical, $V(G)\setminus S_5$ is a 4-chromatic homogeneous set, which contradicts \ref{clm:homo-chi}.

So there is a component $L$ of $S_0$ that is anti-complete to $S_2$. By \ref{clm:0-134}, $N(L)\subseteq S_5$. By \ref{clm:5-stable} and since $G$ is connected and has no clique cutsets by \cref{lem:critical}, $N(L)$ is a stable set on at least two vertices. In all 4-colourings of $G-L$, $N(L)\subseteq S_5$ is monochromatic, because $\chi(G[Q])=3$. Since $G$ is not 4-colourable, this in turn implies that there is no 4-colouring of $G[L\cup N(L)]$ such that $N(L)$ is monochromatic. Let $G'=G[Q\cup N(L)\cup L]$. In all 4-colourings of $G'-L$, $N(L)$ is monochromatic. So $G'$ is not 4-colourable. Since $G$ is 5-vertex-critical, we have $G=G'$. Then $Q$ is a homogeneous $C_5$ in $G$, a contradiction.
\end{proof2}

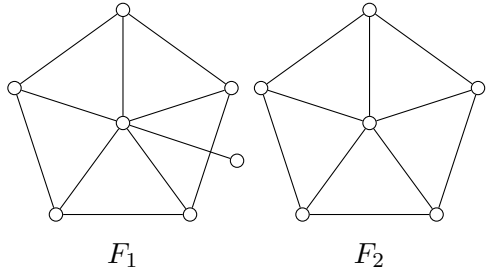
\begin{figure}[h!]
\centering
\begin{tikzpicture}[scale=0.5]
\tikzstyle{vertex}=[circle, draw, fill=white, inner sep=1pt, minimum size=5pt]

    \node at (0,-3.5) {$F_1$};
    \node[vertex](1) at (0,3) {};
    \node[vertex](2) at (-{3*cos(18)},{3*sin(18}){} ;
    \node[vertex](3) at (-{3*sin(36)},-{3*cos(36}) {};
    \node[vertex](4) at ({3*sin(36)},-{3*cos(36}) {};
	\node[vertex](5) at ({3*cos(18)},{3*sin(18}) {};
	\node[vertex](6) at (0,0){} ;
	\node[vertex](7) at (3,-1) {};

    \foreach \from/\to in {1/2,2/3,3/4,4/5,1/5,1/6,2/6,3/6,4/6,5/6,6/7}
		\draw (\from) -- (\to);

\end{tikzpicture}
\begin{tikzpicture}[scale=0.5]
\tikzstyle{vertex}=[circle, draw, fill=white, inner sep=1pt, minimum size=5pt]

    \node at (0,-3.5) {$F_2$};
    \node[vertex](1) at (0,3) {};
    \node[vertex](2) at (-{3*cos(18)},{3*sin(18}){} ;
    \node[vertex](3) at (-{3*sin(36)},-{3*cos(36}) {};
    \node[vertex](4) at ({3*sin(36)},-{3*cos(36}) {};
    \node[vertex](5) at ({3*cos(18)},{3*sin(18}) {};
    \node[vertex](6) at (0,0){} ;

    \foreach \from/\to in {1/2,2/3,3/4,4/5,1/5,1/6,2/6,3/6,4/6,5/6}
		\draw (\from) -- (\to);

\end{tikzpicture}
\caption{The graphs $F_1$ and $F_2$.}\label{fig:F1}
\end{figure}


\item\label{clm:S5-empty}
We may assume that $G$ is $F_2$-free (see \cref{fig:F1} for the graph $F_2$), and $S_5$ is empty.

\begin{proof2}
Let $F_1$ be the first graph in  \cref{fig:F1}.  By \ref{clm:5or0empty} and since $Q$ is an arbitrary induced $C_5$ in $G$, $G$ is $F_1$-free. We run the program for 5-vertex-critical ($P_6$,bull,$F_1$)-free graphs that contain an induced $F_2$. The program terminates and outputs 94 graphs.
\end{proof2}

By \ref{clm:0-134}, \ref{clm:0-2} and \ref{clm:S5-empty}, we have the following \ref{clm:0-homo}.

\item\label{clm:0-homo}
Each component of $S_0$ is homogeneous.

\item\label{clm:2--1-3}
If $v_1\in S_2(i)$ has a nonneighbour $v_2\in S_1(i)$, then $v_1$ is anti-complete to $S_3(i)$.

\begin{proof2}
Suppose that $v_3\in S_3(i)$ is adjacent to $v_1$. By \ref{clm:1-3}, $v_2v_3\in E(G)$. Then $\{v_2,v_3,u_{i+1},u_{i+2},v_1\}$ induces a bull.
\end{proof2}

\item\label{clm:2--0-3}
If $v_1\in S_2(i)$ has a neighbour $v_2\in S_0$, then $v_1$ is anti-complete to $S_3(i)$.

\begin{proof2}
Suppose that $v_3\in S_3(i)$ is adjacent to $v_1$. By \ref{clm:0-134}, $v_2v_3\notin E(G)$. Then $\{v_2,v_1,u_{i+1},u_{i+2},v_3\}$ induces a bull.
\end{proof2}

\vspace{0.5cm}
Let $A(i)=\{v\in V(G)\setminus(Q\setminus\{u_i\}): N_{Q\setminus\{u_i\}}(v)=\{u_{i+1},u_{i-1}\}\}=\{u_i\}\cup S_2(i)\cup S_3(i)$, and $B(i)=\{v\in V(G)\setminus(Q\setminus\{u_i\}): N_{Q\setminus\{u_i\}}(v)=\emptyset\}=S_1(i)\cup S_0$. We use the following \ref{clm:A-2component-homo}-\ref{clm:1-clique} to characterise $A(i)$ and $B(i)$.

\item\label{clm:A-2component-homo}
Suppose that $A(i)$ has at least two components. Then each component of $A(i)$ is homogeneous and is a clique on at most 3 vertices.

\begin{proof2}
Let $X$ be a component of $A(i)$. Since any vertex in  $A(i)$ together with $Q\setminus\{u_i\}$ induces a $C_5$, we may assume that $u_i\notin X$. Then $X$ is a component of $S_2(i)$, and $X$ is anti-complete to $S_3(i)$. By \ref{clm:2-homo}, $X$ is homogeneous. By \ref{clm:homo-chi}, $G[X]$ is a clique on at most 3 vertices.
\end{proof2}

\item\label{clm:A-1component-clique}
Suppose that $A(i)$ has only one component. Then $A(i)$ is a clique on at most 3 vertices.

\begin{proof2}
Since $G[A(i)]$ is connected, for each component $P$ of $S_2(i)$, $P$ has a neighbour in $S_3(i)$. Let $p\in P$ have a neighbour in $S_3(i)$. By  \ref{clm:2--1-3} and \ref{clm:2--0-3}, $p$ is complete to $S_1(i)$ and anti-complete to $S_0$. By \ref{clm:0-2} and \ref{clm:1-2+0}, $P$ is complete to $S_1(i)$ and anti-complete to $S_0$. By \ref{clm:0-134} and \ref{clm:1-3}, $A(i)$ is complete to $S_1(i)$ and anti-complete to $S_0$. 
By \ref{clm:1-2+1} and \ref{clm:1-3}, $S_1(i+1)\cup S_1(i-1)$ is anti-complete to $A(i)$. By \ref{clm:2-2+2} and \ref{clm:2-3+12}, $S_2(i+2)\cup S_2(i-2)$ is anti-complete to $A(i)$. By \ref{clm:2-3+12} and \ref{clm:3-3+1}, $S_3(i+1)\cup S_3(i-1)$ is complete to $A(i)$. By \ref{clm:S5-empty}, $S_5$ is empty. Suppose that $v\in S_1(i+2)\cup S_1(i-2)\cup S_2(i+1)\cup S_2(i-1)$ is mixed on an edge $a_1a_2$ in $A(i)$. By symmetry we may assume that $N_{Q\setminus\{u_i\}}(v)=\{u_{i+2}\}$. Then $\{a_1,a_2,u_{i-1},u_{i-2},v\}$ induces a bull. So each vertex in $S_1(i+2)\cup S_1(i-2)\cup S_2(i+1)\cup S_2(i-1)$ is pure to $A(i)$.
If $v\in S_4(i)$ is adjacent to $a\in A(i)$, then $\{v,a,u_{i+1},u_{i+2},u_{i-2},u_{i-1}\}$ induces an $F_2$, which contradicts \ref{clm:S5-empty}. So $S_4(i)$ is anti-complete to $A(i)$. Suppose that $v\in S_4(i+2)\cup S_4(i-2)$  has a nonneighbour $a$ in $A(i)$. By symmetry, let $v\in S_4(i-2)$. Then $\{u_{i+1},u_{i+2},v,a,u_{i-2}\}$ induces a bull. So $S_4(i+2)\cup S_4(i-2)$ is complete to $A(i)$.

By the above paragraph and by \ref{clm:whole}, if $v\in V(G)\setminus A(i)$ is mixed on $A(i)$, then $v\in S_3(i+2)\cup S_3(i-2)\cup S_4(i+1)\cup S_4(i-1)$. Suppose that $a_1,a_2\in A(i)$ are nonadjacent, and $v$ is adjacent to $a_1$ and nonadjacent to $a_2$. Note that $N_{Q\setminus\{u_i\}}(v)=\{u_{i+1},u_{i+2},u_{i-2}\}$ or $\{u_{i+2},u_{i-2},u_{i-1}\}$.
By symmetry let $N_{Q\setminus\{u_i\}}(v)=\{u_{i+1},u_{i+2},u_{i-2}\}$. Then $\{v,a_1,u_{i+1},a_2,u_{i-2}\}$ induces a bull. So for any pair of nonadjacent vertices in $A(i)$, their neighbourhoods in $G-A(i)$ are the same. Let $M$ be a component of $\overline{G[A_i]}$, then the vertices in $M$ have the same neighbourhood in $G-A(i)$, and $V(M)$ is complete to $A_i\setminus V(M)$ in $G$. So $V(M)$ is homogeneous in $G$. By \ref{clm:homo-chi} and since $\overline{G[V(M)]}$ is connected, $M$ is a single vertex. Since each component of $\overline{G[A_i]}$ is a single vertex, $A(i)$ is a clique in $G$. Since $u_{i+1}$ is complete to $A(i)$, and $G$ is $K_5$-free, $A(i)$ has at most 3 vertices.
\end{proof2}

By \ref{clm:A-2component-homo} and \ref{clm:A-1component-clique} we have the following \ref{clm:2-3+0} and \ref{clm:3-clique}.

\item\label{clm:2-3+0}
$S_2(i)$ is anti-complete to $S_3(i)$.

\item\label{clm:3-clique}
$S_3(i)$ is a clique on at most 2 vertices.

By \ref{clm:homo-chi}, \ref{clm:0-134}, \ref{clm:1-homo}, \ref{clm:0-homo}, \ref{clm:A-2component-homo} and \ref{clm:A-1component-clique}, we have the following \ref{clm:AB-pure}.

\item\label{clm:AB-pure}
Each component of $B(i)$ is a component of $S_1(i)$ or $S_0$. If $A(i)$ has at least two components, then each component of $A(i)$ or $B(i)$ is homogeneous, and is a clique on at most 3 vertices. 

\item\label{clm:AB-neighbourhood}
Suppose that $A(i)$ has at least two components. Let $X\cup Y$ be a component of $A(i)\cup B(i)$, with $X\subseteq A(i)$ and $Y\subseteq B(i)$. Then any $v\in V(G)\setminus(X\cup Y)$ is pure to $X$ and pure to $Y$.

\begin{proof2}
Since $A(i)$ has at least two components, we have $S_2(i)\neq\emptyset$. Thus, $S_4\setminus S_4(i)$ is empty by \ref{clm:2-4+12-2}. If $S_4(i)$ and $S_3(i)$ are non-empty and $v_1\in S_4(i)$ and $v_2\in S_3(i)$ are adjacent, then $\{v_1,v_2\}\cup(Q\setminus\{u_i\})$ induces an $F_2$, which contradicts \ref{clm:S5-empty}. So $S_4(i)$ is anti-complete to $S_3(i)$. By \ref{clm:whole}, \ref{clm:1-2+1}, \ref{clm:1-3}, \ref{clm:2-2+2}, \ref{clm:2-3+12}, \ref{clm:2-4+0}, \ref{clm:3-3+1} and \ref{clm:S5-empty}, if $v\in V(G)\setminus(X\cup Y)$ is mixed on $X$, then $v\in S_1(i+2)\cup S_1(i-2)\cup S_2(i+1)\cup S_2(i-1)\cup S_3(i+2)\cup S_3(i-2)$. By \ref{clm:whole}, \ref{clm:0-134}, \ref{clm:1-1}, \ref{clm:1-3}, \ref{clm:1-4+0} and \ref{clm:S5-empty}, if $v\in V(G)\setminus(X\cup Y)$ is mixed on $Y$, then $v\in S_1(i+2)\cup S_1(i-2)\cup S_2(i+1)\cup S_2(i-1)\cup S_2(i+2)\cup S_2(i-2)$.

Suppose that $v\in S_2(i+2)\cup S_2(i-2)$. By symmetry let $v\in S_2(i+2)$. Suppose that $y_1,y_2\in Y$ are such that $v$ is adjacent to $y_1$ and nonadjacent to $y_2$. By \ref{clm:AB-pure}, $y_1$ and $y_2$ are in different components of $Y$. By \ref{clm:2-2+2} and \ref{clm:2-3+12}, $v$ is anti-complete to $X$. Since $X\cup Y$ is connected, by \ref{clm:AB-pure} and since $G$ is $P_6$-free, there is an induced $P_3$ or $P_5$ in $X\cup Y$ such that $y_1$ and $y_2$ are degree 1 vertices of the path, and each edge in the path has an end in $X$ and an end in $Y$. If the path is $P_3=y_1x_1y_2$, then $\{y_2,x_1,y_1,v,u_{i-2},u_{i+2}\}$ induces a $P_6$. If the path is $P_5=y_1x_1y_3x_2y_2$, then $\{y_2,x_2,y_3,v,u_{i-2},u_{i+2}\}$ or $\{y_3,x_1,y_1,v,u_{i-2},u_{i+2}\}$ induces a $P_6$, depending on whether $v$ is adjacent to $y_3$.

Suppose that $v\in S_3(i+2)\cup S_3(i-2)$. By symmetry let $v\in S_3(i+2)$. Suppose that $x_1,x_2\in X$ are such that $v$ is adjacent to $x_1$ and nonadjacent to $x_2$. By \ref{clm:AB-pure}, $x_1$ and $x_2$ are nonadjacent. Then $\{x_1,u_{i+1},v,x_2,u_{i-2}\}$ induces a bull.

Suppose that $v\in S_1(i+2)\cup S_1(i-2)\cup S_2(i+1)\cup S_2(i-1)$. Then $N_{Q\setminus\{u_i\}}(v)=\{u_{i+2}\}$ or $\{u_{i-2}\}$. By symmetry let $v$ be adjacent to $u_{i+2}$. Let $xy$ be an edge in $X\cup Y$ with $x\in X$ and $y\in Y$. If $x,y$ are adjacent to $v$, then $\{v,x,y,u_{i+2},u_{i-1}\}$ induces a bull. If $x,y$ are nonadjacent to $v$, then $\{y,x,u_{i-1},u_{i-2},u_{i+2},v\}$ induces a $P_6$. So $v$ is adjacent to exactly one of $x$ and $y$. Let $x_1,x_2\in X$. If $x_1,x_2$ are in the same component of $A(i)$, then $v$ is pure to $\{x_1,x_2\}$ since each component of $A(i)$ is homogeneous. If $x_1,x_2$ are in different components of $A(i)$, by \ref{clm:AB-pure} and since $X\cup Y$ is connected, there is an induced path in $X\cup Y$ such that $x_1$ and $x_2$ are degree 1 vertices of the path, and each edge in the path has an end in $X$ and an end in $Y$. The path has an even number of edges, so $v$ is pure to $\{x_1,x_2\}$. This proves that $v$ is pure to $X$. By the same argument, $v$ is pure to $Y$.
\end{proof2}

\item\label{clm:AB-num-component}
Suppose that $A(i)$ has at least two components. Let $X\cup Y$ be a component of $A(i)\cup B(i)$, with $X\subseteq A(i)$ and $Y\subseteq B(i)$. Then each of $X$ and $Y$ has at most one component.

\begin{proof2}
In this proof we only use \ref{clm:AB-pure} and \ref{clm:AB-neighbourhood}, and the fact that $G$ is 5-vertex-critical and $K_5$-free, so the positions of $X$ and $Y$ are symmetric. If $X$ or $Y$ is empty, since $X\cup Y$ is connected, the other set has at most one component. So  we may assume that both $X$ and $Y$ are nonempty.

Suppose that $X$ or $Y$ has only one component, say $X$. Suppose that $Y_1$ and $Y_2$ are components of $Y$. Without loss of generality, we may assume that $\chi(Y_1)\leq \chi(Y_2)$. Then $Y_1$ is anti-complete to $Y_2$, and $Y_2$ is complete to $N(Y_1)$ by \ref{clm:AB-neighbourhood}. This contradicts \cref{lem:comparable}, which implies that $Y$ has only one component. So we may assume that both $X$ and $Y$ have at least two components. Let $\alpha_1,\alpha_2,\alpha_3,\alpha_4$ be 4 colours. We prove that in all cases, a 4-colouring of $G$ can be obtained from a 4-colouring of a proper induced subgraph of $G$, which contradicts the fact that $G$ is 5-vertex-critical.

First assume that there is a $K_4=X_1\cup Y_1$ in $X\cup Y$, with $X_1\subseteq X$ and $Y_1\subseteq Y$. Let $|X_1|=m$ and $|Y_1|=4-m$. Without loss of generality, we may assume that $m\geq 2$. Since $X\cup Y$ is connected, every 3-vertex component of $X$ has a neighbour in $Y$. So we may choose a 3-vertex component of $X$ to be $X_1$ if this exists. Thus, we may assume that $X_1$ is the largest component of $X$.  We may give $G-(X\setminus X_1)$ a 4-colouring. Assume that $X$ receives colours $\alpha_1,\ldots,\alpha_m$. Then $N(X)$ receives colours $\alpha_{m+1},\ldots,\alpha_4$. Since $Y_1$ uses all colours of $\alpha_{m+1},\ldots,\alpha_4$, these colours are not used in $N(Y)\setminus X$. Let $n$ be the size of the largest component of $Y$. Then there are at most $4-n$ colours that are used in $N(Y)\setminus X$. We may assume that the colours $\alpha_{5-n},\ldots,\alpha_4$ are not used in $N(Y)\setminus X$. Then we may change the colouring of $Y$ such that each $p$-vertex component of $Y$ receives colours $\alpha_{5-p},\ldots,\alpha_4$. Then we may colour each $q$-vertex component of $X\setminus X_1$ with $\alpha_1,\ldots,\alpha_q$ without conflict since $G$ is $K_5$-free. This gives a 4-colouring of $G$.

Then assume that $X\cup Y$ is $K_4$-free and has a $K_3=X_1\cup Y_1$, with $X_1\subseteq X$ and $Y_1\subseteq Y$. We may assume that $|X_1|=2$ and $|Y_1|=1$. Since $X\cup Y$ is connected, each of $X$ and $Y$ is $K_3$-free. We may give a 4-colouring of $G-(X\setminus X_1)$. Assume that $X$ receives colours $\alpha_1,\alpha_2$, and $Y_1$ receives colour $\alpha_3$. If $Y$ has no 2-vertex components, then we may change the colouring of $Y\setminus Y_1$ to $\alpha_3$, and colour $X\setminus X_1$ with colours $\alpha_1,\alpha_2$. So we may assume that $Y$ has a 2-vertex component $Y_2$, and $Y_2$ is anti-complete to $X_1$ since $X\cup Y$ is $K_4$-free. Since $N(Y)\setminus X$ does not use $\alpha_3$, we may use $\alpha_3$ to colour a vertex in $Y_2$. Then one of $\alpha_1$ and $\alpha_2$, say $\alpha_1$, is not used in $Y_2$. Then we may change the colouring of $Y\setminus (Y_1\cup Y_2)$ such that each $p$-vertex component of $Y\setminus (Y_1\cup Y_2)$ uses the same colours as $Y_p$, and we may colour each $q$-vertex component of $X\setminus X_1$ with $\alpha_1,\ldots,\alpha_q$. This gives a 4-colouring of $G$.

Finally assume that $X\cup Y$ is $K_3$-free. Then each component of $X$ or $Y$ has only one vertex. Let $xy$ be an edge in $X\cup Y$ with $x\in X$ and $y\in Y$. We may give a 4-colouring of $G-(X\cup Y\setminus\{x,y\})$, then colour $X\setminus\{x\}$ with the same colour of $x$ and colour $Y\setminus\{y\}$ with the same colour of $y$. This gives a 4-colouring of $G$.
\end{proof2}

\item\label{clm:AB-num-component-2}
If $A(i)$ has only one component, then $S_1(i)$ is a clique on at most 3 vertices.

\begin{proof2}
Suppose that $A(i)$ has only one component, by \ref{clm:A-1component-clique} and since $u_i\in A(i)$, we have that $S_2(i)$ is empty. By \ref{clm:1-homo} and \ref{clm:1-private}, $S_1(i)$ is homogeneous in $G$. By \ref{clm:homo-chi}, $S_1(i)$ is a clique on at most 3 vertices.
\end{proof2}

By \ref{clm:AB-num-component} and \ref{clm:AB-num-component-2} we have the following \ref{clm:1-2+0-new} and \ref{clm:1-clique}.

\item\label{clm:1-2+0-new}
$S_1(i)$ is anti-complete to $S_2(i)$.

\item\label{clm:1-clique}
$S_1(i)$ is a clique on at most 3 vertices.

Let $F_3$-$F_6$ be given in \cref{fig:F3-F6}. By \ref{clm:2-3+0}, \ref{clm:3-clique}, \ref{clm:1-2+0-new} and \ref{clm:1-clique}, and since $Q$ is an arbitrary induced $C_5$ in $G$, we have the following \ref{clm:F3-F5-free}.

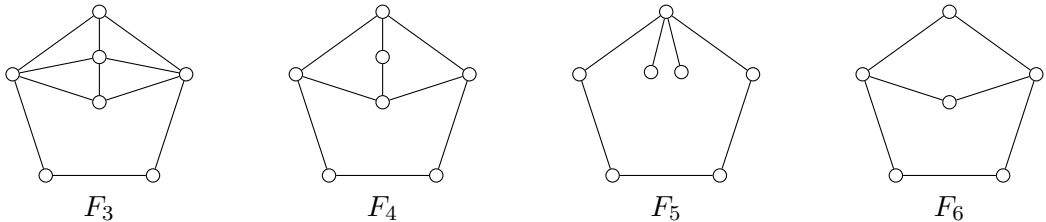
\begin{figure}[h!]
\centering
\begin{tikzpicture}[scale=0.4]
\tikzstyle{vertex}=[circle, draw, fill=white, inner sep=1pt, minimum size=5pt]

	\node at (0,-3.5) {$F_3$};
	\node[vertex](1) at (0,3) {};
	\node[vertex](2) at (-{3*cos(18)},{3*sin(18}) {};
	\node[vertex](3) at (-{3*sin(36)},-{3*cos(36}) {};
   	\node[vertex](4) at ({3*sin(36)},-{3*cos(36}) {};
	\node[vertex](5) at ({3*cos(18)},{3*sin(18}) {};
	\node[vertex](6) at (0,0) {};
	\node[vertex](7) at (0,1.5) {};

    \foreach \from/\to in {1/2,2/3,3/4,4/5,1/5,2/6,5/6,1/7,2/7,5/7,6/7}
		\draw (\from) -- (\to);
\end{tikzpicture}
\hspace{10mm}
\begin{tikzpicture}[scale=0.4]
\tikzstyle{vertex}=[circle, draw, fill=white, inner sep=1pt, minimum size=5pt]

	\node at (0,-3.5) {$F_4$};
	\node[vertex](1) at (0,3) {};
	\node[vertex](2) at (-{3*cos(18)},{3*sin(18}) {};
	\node[vertex](3) at (-{3*sin(36)},-{3*cos(36}) {};
   	\node[vertex](4) at ({3*sin(36)},-{3*cos(36}) {};
	\node[vertex](5) at ({3*cos(18)},{3*sin(18}) {};
	\node[vertex](6) at (0,0) {};
	\node[vertex](7) at (0,1.5) {};

    \foreach \from/\to in {1/2,2/3,3/4,4/5,1/5,2/6,5/6,1/7,6/7}
		\draw (\from) -- (\to);
\end{tikzpicture}
\hspace{10mm}
\begin{tikzpicture}[scale=0.4]
\tikzstyle{vertex}=[circle, draw, fill=white, inner sep=1pt, minimum size=5pt]

	\node at (0,-3.5) {$F_5$};
	\node[vertex](1) at (0,3) {};
	\node[vertex](2) at (-{3*cos(18)},{3*sin(18}) {};
	\node[vertex](3) at (-{3*sin(36)},-{3*cos(36}) {};
   	\node[vertex](4) at ({3*sin(36)},-{3*cos(36}) {};
	\node[vertex](5) at ({3*cos(18)},{3*sin(18}) {};
	\node[vertex](6) at (-0.5,1) {};
	\node[vertex](7) at (0.5,1) {};

    \foreach \from/\to in {1/2,2/3,3/4,4/5,1/5,1/6,1/7}
		\draw (\from) -- (\to);
\end{tikzpicture}
\hspace{10mm}
\begin{tikzpicture}[scale=0.4]
\tikzstyle{vertex}=[circle, draw, fill=white, inner sep=1pt, minimum size=5pt]

	\node at (0,-3.5) {$F_6$};
	\node[vertex](1) at (0,3) {};
	\node[vertex](2) at (-{3*cos(18)},{3*sin(18}) {};
	\node[vertex](3) at (-{3*sin(36)},-{3*cos(36}) {};
   	\node[vertex](4) at ({3*sin(36)},-{3*cos(36}) {};
	\node[vertex](5) at ({3*cos(18)},{3*sin(18}) {};
	\node[vertex](6) at (0,0) {};

    \foreach \from/\to in {1/2,2/3,3/4,4/5,1/5,2/6,5/6}
		\draw (\from) -- (\to);
\end{tikzpicture}
\caption{The graphs $F_3$-$F_6$.}\label{fig:F3-F6}
\end{figure}

\item\label{clm:F3-F5-free}
$G$ is $(F_3,F_4,F_5)$-free.

\item\label{clm:F6-free}
We may assume that $G$ is $F_6$-free.

\begin{proof2}
We run the program for 5-vertex-critical ($P_6$,bull,$F_2$,$F_3$,$F_4$,$F_5$)-free graphs that contain an induced $F_6$. The program terminates and outputs 25 graphs.
\end{proof2}

\end{enumerate}

By \ref{clm:S5-empty} and \ref{clm:F6-free}, $S_2$ and $S_5$ are empty. By \ref{clm:0-134} and since $G$ is connected by \cref{lem:critical}, $S_0$ is empty. By \ref{clm:4-size}, $|S_4(i)|\leq 2$. By \ref{clm:3-clique}, $|S_3(i)|\leq 2$. By \ref{clm:1-clique}, $|S_1(i)|\leq 3$. So $|V(G)|$ is bounded.
\end{proof}

Now we prove the main theorem. The main idea is that every 5-vertex-critical ($P_6$,bull)-free graph can be obtained from a 5-vertex-critical ($P_6$,bull)-free graph with no homogeneous $C_5$ by a bounded number of operations of substituting $C_5$ for a homogeneous $K_3$.

\begin{theorem}\label{thm:main}
There are finitely many 5-vertex-critical ($P_6$,bull)-free graphs.
\end{theorem}
\begin{proof}
Let $G$ be a 5-vertex-critical ($P_6$,bull)-free graph containing a homogenous $C_5$, which we call $S$. Let $G'$ be the graph obtained from $G$ by substituting $S'$ for $S$, where $S'=K_3$. By \cref{lem:substitute}, $G'$ is 5-vertex-critical.

Suppose that $F\subseteq V(G')$ induces a $P_6$ or bull. Since $S'$ is homogeneous, and the graphs $P_6$ and bull have no nontrival homogeneous sets, we have that $|F\cap S'|\leq 1$. If $F \subseteq V(G'-S')=V(G-S)$, then $G$ has an induced $P_6$ or bull, a contradiction. Let $\{v'\}=F\cap S'$. Let $v\in S$. Since $N_{G'}(v')=N_G(v)$, $(F\setminus \{v'\})\cup \{v\}$ must be isomorphic to $F$, but then $G$ has an induced $P_6$ or bull, a contradiction. So $G'$ is ($P_6$,bull)-free. By the same argument, the reverse operation (i.e., substituting $C_5$ for a homogeneous $K_3$) does not create an induced $P_6$ or bull.

In the above operation, $|G'|=|G|-2$. So for any 5-vertex-critical ($P_6$,bull)-free graph $G$ with a homogeneous $C_5$, one can do the operation of substituting $K_3$ for a homogeneous $C_5$ for at most $|G|/2$ times and end with a 5-vertex-critical ($P_6$,bull)-free graph with no homogeneous $C_5$. Since the substitution operation is reversible, any 5-vertex-critical ($P_6$,bull)-free graph with a homogeneous $C_5$ can be obtained from a 5-vertex-critical ($P_6$,bull)-free graph with no homogeneous $C_5$ by a series of operations of substituting $C_5$ for a homogeneous $K_3$.

Let $G'$ be a 5-vertex-critical ($P_6$,bull)-free graph containing a homogeneous $K_3$, which we call $S'$. Let $G$ be the graph obtained from $G'$ by substituting $S$ for $S'$, where $S=C_5$. Now we prove that $G$ has fewer occurrences of a homogeneous $K_3$ than $G'$, which implies that a 5-vertex-critical ($P_6$,bull)-free graph $H$ can do the operation of substituting $C_5$ for a homogeneous $K_3$ by a bounded number of times, namely at most the number of occurrences of a homogeneous $K_3$ of $H$.

Suppose that $F\subseteq V(G)$ is a homogeneous $K_3$ in $G$ but not a homogeneous $K_3$ in $G'$. If $F\subseteq V(G-S)$, then $F\subseteq N_G(S)=N_{G'}(S')$ or $F\subseteq \overline{N_G(S)}=\overline{N_{G'}(S')}$, so $F$ is a homogeneous $K_3$ in $G'$, a contradiction. So $F\cap S\neq\emptyset$. Since $\omega(S)=2$, $|F\cap S|=1$ or 2. Let $S=\{v_1,v_2,v_3,v_4,v_5\}$ with $v_iv_{i+1}\in E(G)$, with all indices modulo 5. We may assume that $F\cap S=\{v_1\}$ or $\{v_1,v_2\}$. If $F\cap S=\{v_1,v_2\}$, then $v_3$ is mixed on $F$. If $F\cap S=\{v_1\}$, then $F\setminus S\subseteq N_G(S)$, and then $v_3$ is mixed on $F$. So all occurrences of a homogeneous $K_3$ in $G$ are also a homogeneous $K_3$ in $G'$. Since $S'$ is a homogeneous $K_3$ in $G'$ but not in $G$, $G$ has fewer occurrences of a homogeneous $K_3$ than $G'$.

By \cref{thm:no-homo-C5-finite}, we may assume that there exist integers $n$ and $m$ such that every 5-vertex-critical ($P_6$,bull)-free graph with no homogeneous $C_5$ has at most $n$ vertices, and has at most $m$ occurrences of a homogeneous $K_3$. Every 5-vertex-critical ($P_6$,bull)-free graph with a homogeneous $C_5$ can be obtained from a 5-vertex-critical ($P_6$,bull)-free graph with no homogeneous $C_5$ by a series of operations of substituting $C_5$ for a homogeneous $K_3$, and so has at most $n+2m$ vertices.
\end{proof}

\subsection*{Acknowledgements}

\noindent  Shenwei Huang is supported by the Natural Science Foundation of China (NSFC) under Grant 12171256 and 12161141006.
Jan Goedgebeur is supported by Internal Funds of KU Leuven and a grant of the Research Foundation Flanders (FWO) with grant number G0AGX24N. Jorik Jooken is supported by a Postdoctoral Fellowship of the FWO with grant number 1222524N.
We also acknowledge the support of the joint NSFC-FWO scientific mobility project with grant number 12311530678 and the support of the joint FWO-NSFC scientific mobility project with grant number VS01224N.

\end{document}